\documentclass[10pt]{amsart}
\usepackage{bm}
\usepackage{pstricks}
\usepackage{graphicx}
\usepackage{float}
\usepackage{subfig}
\usepackage{cite}

\newcommand{\bR}{\mathbb R}

\theoremstyle{plain}
\newtheorem{theorem}{Theorem}[section]

\newtheorem{conjecture}[theorem]{Conjecture}

\theoremstyle{definition}

\newtheorem{definition}[theorem]{Definition}
\newtheorem{remark}[theorem]{Remark}

\author{Patrick Dukes}
\address{ Department of Mathematics ,Winthrop University,Rock Hill SC
2973, USA} 
\email{pdukes3@gmail.com}
\author{Joe Rusinko}
\address{Department of Mathematics, Winthrop University, Rock Hill SC
29733, USA}
\email{rusinkoj@winthrop.edu}

\title[Double Wiring]{Commutation Classes of Double Wiring Diagrams}
\date\today

\begin{document}
\begin{abstract} We describe a new method for computing the graph of commutation classes of double wiring diagrams. Using these methods we compute the graph for five strings or less which allows us to confirm a positivity conjecture of Fomin and Zelevinsky when $n\le4$ . 
\end{abstract}
\maketitle
\section{Introduction}
\label{sec:intro}

Double wiring diagrams were introduced to study totally positive matrices and became a motivating example in the study of cluster algebras (\protect\cite{cluster1},\protect\cite{laurent},\protect\cite{TP}). In particular, the graph displaying the relationships among the commutation classes is a precursor to the exchange graph, and the relationships among the chamber minors are precursors to exchange relations.

In \protect\cite{TP} Fomin and Zelevinsky define an \emph{n-stringed double wiring diagram} as
two sets of $n$ piecewise linear lines (red and blue) such that each line intersects every other line of the same color exactly once. We number red lines from $1$ to $n$ with $1$ on the top left and $n$ on the bottom left. The blue lines are labeled in the reverse order. In addition, every chamber of the wiring diagram is labeled with a pair of subsets $(r,b)$ where $r$ (resp. $b$) is the subset of $\{1,2,\cdots ,n\}$ identifying the red (resp. blue) strings which pass below the chamber. See Figure \protect\ref{4wire} for an example.
\begin{figure}[H]
\includegraphics[width=1\linewidth]{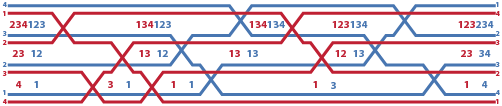}
\caption{A four string double wiring diagram with chamber labels.}
\label{4wire}
\end{figure}
It is possible that slightly different wiring diagrams yield the same collection of chamber labels.  
Fomin and Zelevinsky consider two wiring diagrams which share the same collection of chamber labels \emph{isotopic}.  For single wiring diagrams such collections of diagrams are called commutation classes which have been studied in   \protect\cite{single2} \protect\cite{single1}.  To keep this connection clear, we use the term  \emph{commutation classes of double wiring diagrams}.

Any two commutation classes of double wiring diagrams can be linked by a sequence of the braid moves pictured in Figure \protect\ref{moves} \protect\cite{TP}.  Note that in each exchange only one chamber label changes. We will call this label the \emph{center of the braid move}. 
\begin{figure}[H]
   \centering
  \subfloat[2-Move]{\label{fig:2move}\includegraphics[width=0.4\textwidth]{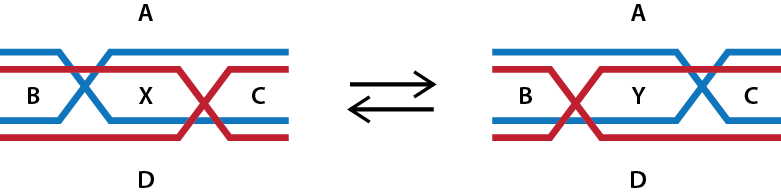}}   \hspace{.2in}
  \subfloat[3-Move]{\label{fig:3move}\includegraphics[width=0.4\textwidth]{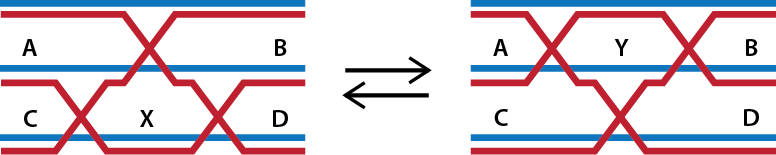}}
  \caption{Braid Moves}
  \label{moves}
\end{figure}
\begin{definition} The graph of commutation classes of wiring diagrams, $\Phi_n$, has a unique vertex for every commutation class of double wiring diagram with $n$ strings.
Two vertices are connected by an edge if their wiring diagrams differ by a single braid move.
\end{definition}
Fomin and Zelevinsky prove that $\Phi_n$ is a finite connected graph and compute $\Phi_3$ \protect\cite{TP}.  In this paper we present a method for computing $\Phi_n$ and use it to construct $\Phi_4$ and $\Phi_5$.   We use these calculations to verify a positivity conjecture of Fomin and Zelevinsky when $n \le 4$.

\section{Computing $\Phi_n$}
\label{sec:compute}
\begin{definition} The quiver $Q(w)$ has vertices corresponding to chamber labels  and an arrow from $(r,b)$ to $(r',b')$ if $r'=r\cup \{r_j\} $ and $b'=b\cup\{b_k\}$ for $r_j,b_k \in \{1,2,\cdots,n\}$.  
\end{definition}
Figure \protect\ref{quiver} shows $Q(w)$ for the wiring diagram pictured in Figure \protect\ref{4wire}.
As appropriate we may label the arrows of the quiver with the pair of numbers being adjoined to $r$ and $b$, or simply by the red or blue numbers individually.
\begin{figure}[H]
\includegraphics[width=\textwidth]{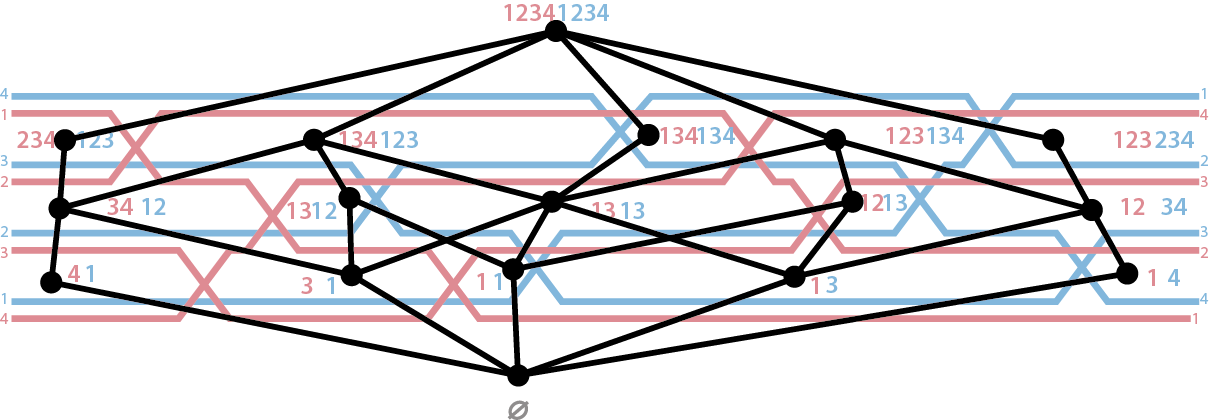}
\caption{Quiver diagram}
\label{quiver}
\end{figure}

\begin{definition} A subquiver is \emph{complete} if it contains every arrow of $Q$ which connects two vertices in the subquiver.
\end{definition}
\begin{definition} A subquiver is \emph{full} if it contains every vertex of $Q$ which lies within the boundary of the subquiver.
\end{definition}

\begin{theorem}
\label{thm:3} There exists a 3-move centered at label $(r,b)$ if and only if $Q(w)$ contains one of the two complete, full subquiver shown in Figure \protect\ref{3quiver}.
\begin{figure}
  \centering
  \subfloat[]{\label{3move1}\includegraphics[width=0.3\textwidth]{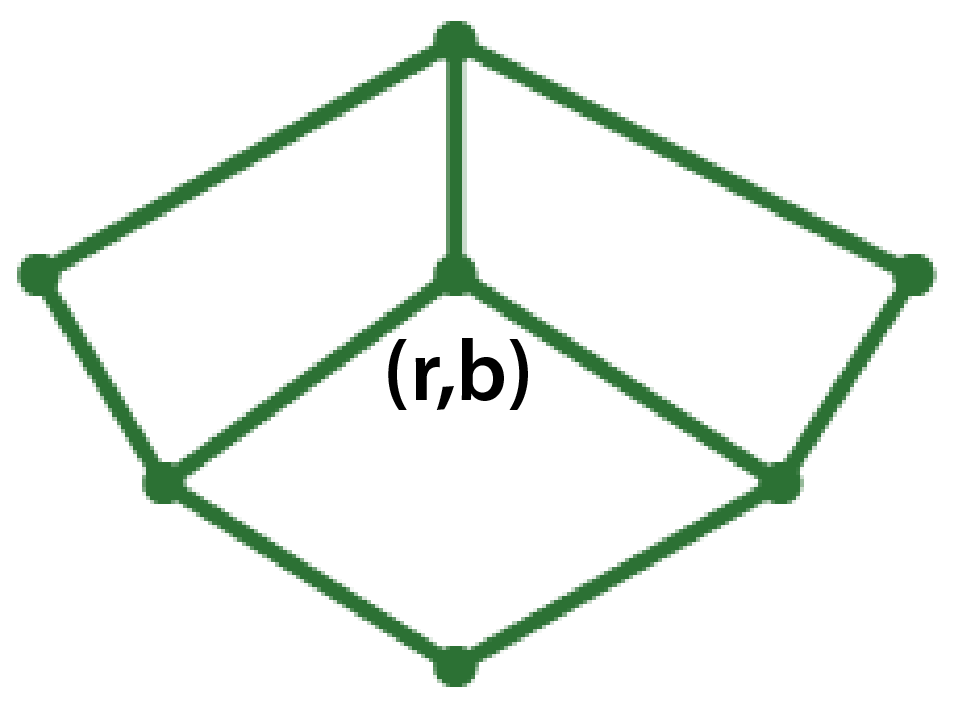}}   
  \subfloat[]{\label{3move2}\includegraphics[width=0.3\textwidth]{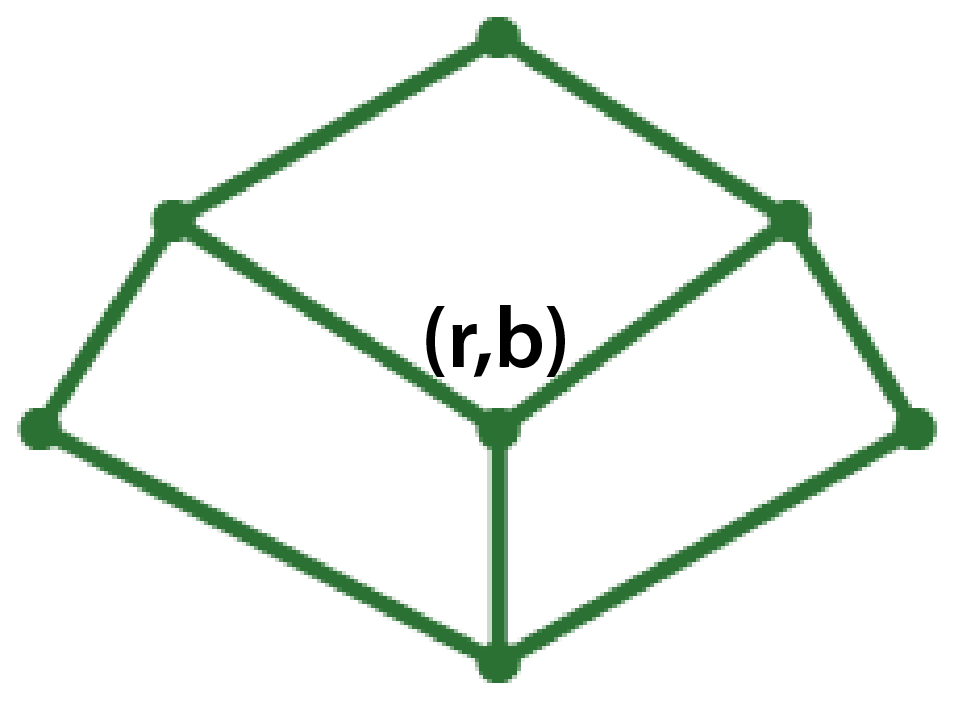}}
  \caption{Subquivers for 3-move}
\label{3quiver}
\end{figure}
\end{theorem}
 \begin{proof}Assume a 3-move exists. Then there must be a region of the wiring diagram isomorphic to Figure \protect\ref{fig:3move}. Constructing the subquiver from this diagram yields Figure \protect\ref{3quiver}.

Assume $Q(w)$ has a compete full subquiver isomorphic to  Figure \protect\ref{3move1} .  We examine the possible red labels for this subquiver. Since the bottom vertex is connected to the top by a path of length three, we know that only three distinct edge labels may appear in this subquiver. We label the left most path from the bottom to the top which passes through $(r,b)$ $x,y,z$ as pictured in Figure \protect\ref{3step1}. 
\begin{figure}[H]
\includegraphics{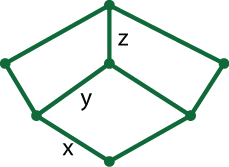}
\caption{}
\label{3step1}
\end{figure}

For each four-cycle in Figure \ref{3step1} only two distinct edge labels may be used since the bottom and top vertices are connected by a path of length two. This limits the potential labelings to those in Figure \protect\ref{3step2}. The case corresponding to picture $d)$ in the figure can not exist because strings  $z$ and $y$ are exchanged  twice which contradicts the definition of  a double wiring diagram.
\begin{figure}[H]
\includegraphics[width=.9\textwidth]{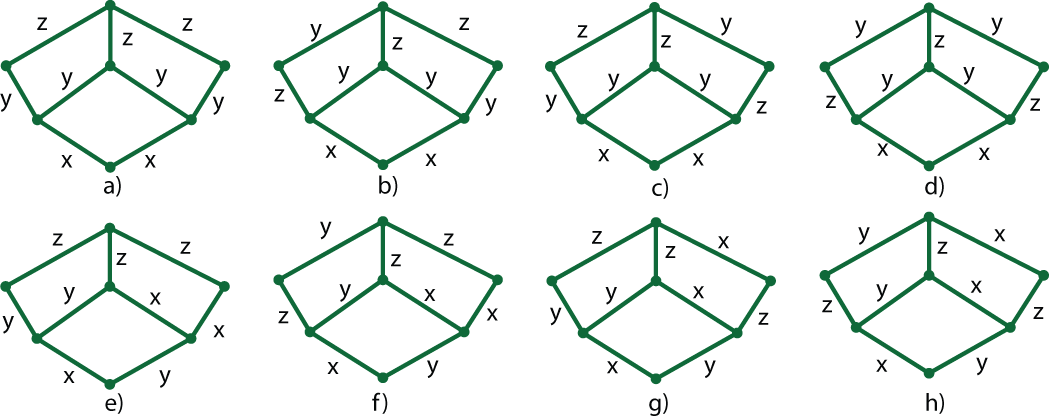}
\caption{Labeled subquivers}
\label{3step2}
\end{figure}

Repeat this argument for the blue strings and label those cases A through H. We now  determine which red and blue cases can be paired together. Since the labels must be distinct the only potential pairs are $(a,H),(b,G)$, and $(c,F)$, and their opposites $(h,A)$,$(g,B)$ and $(f,C)$.  

If we draw a subquiver with the labels in the case $(b,G)$, as in Figure \ref{3step4}, we recover an extra arrow which contradicts the hypothesis that the subquiver was complete.  The pairs  $(c,F),(g,B)$ and $(f,C)$ are symmetric to $(b,G)$ so they are also eliminated. This leaves only $(a,H)$ and $(h,A)$ as possible labelings.
\begin{figure}[H]
\includegraphics[scale=.6]{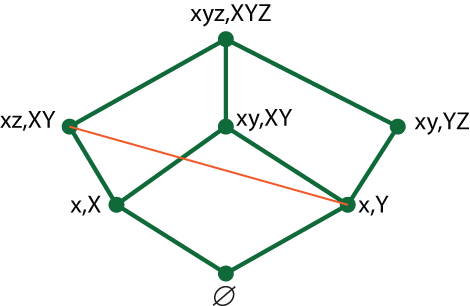}
\caption{}
\label{3step4}
\end{figure}
By symmetry of the labelings we may assume the edge labels are of type $(h,A)$. Since this subquiver is full, there are no missing vertices.  This means that changes in chamber labels of the same cardinality indicate a unique braid crossing as pictured in Figure \ref{3cross}.  
\begin{figure}[H]
  \centering
  \subfloat[]{\label{3cross}\includegraphics[width=0.4\textwidth]{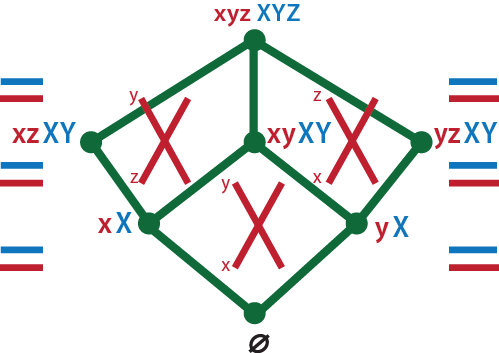}}  \hspace{.3in}
  \subfloat[]{\label{3rebuilt}\includegraphics[width=0.4\textwidth]{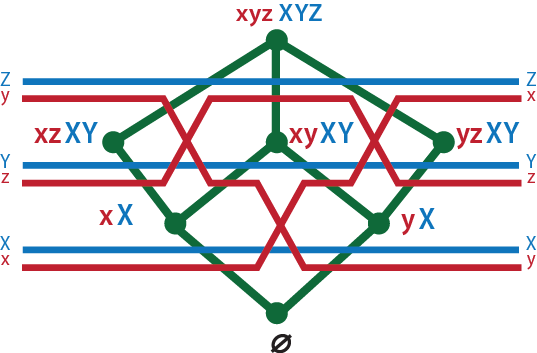}}
  \caption{Reconstructed 3-move}
\label{3recon}
\end{figure}
No other crossings may occur in this region because the quiver is complete.  Therefore the strings  must connect without creating any other crossings. This yields the 3-move pictured in Figure \ref{3rebuilt}.

The proof for  Figure \protect\ref{3move2} follows the same argument with reflected labels.
\end{proof}
\newpage
\begin{theorem}
\label{thm:4} There exists a two move centered at label $(r,b)$ if and only if $Q(w)$ contains the full subquiver shown in Figure \protect\ref{quiver2}.
\begin{figure}[H]
\includegraphics[scale=.8]{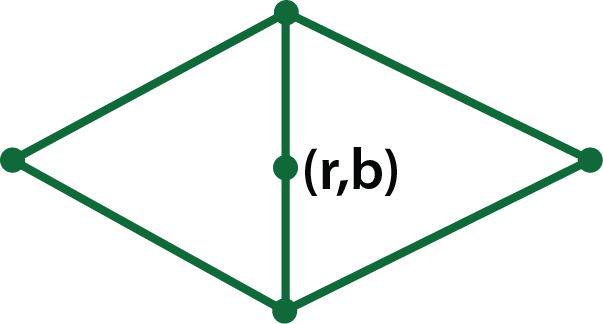}
\caption{Subquiver for 2-move}
\label{quiver2}
\end{figure}
\end{theorem}
\begin{proof}
Assume a two move exists. Then there must be a region of the wiring diagram isomorphic to Figure \ref{fig:2move}. Constructing the quiver from this diagram yields the subquiver in Figure \ref{quiver2}.

Now assume $Q(w)$ contains the full subquiver in Figure \ref{quiver2}. We examine the possible red labels for the subquiver. Label the arrows to and from $(r,b)$ as $x$ and $y$. Since there is a path from the bottom vertex to the top vertex of length two, all arrows in the subquiver  must be labeled $x$ or $y$. Figure \ref{2step1} shows the possible labelings.  The case corresponding to picture $d)$ can be eliminated because it would require  strings $x$ and $y$ to be exchanged  twice.
\begin{figure}[H]
\includegraphics[width=1\textwidth]{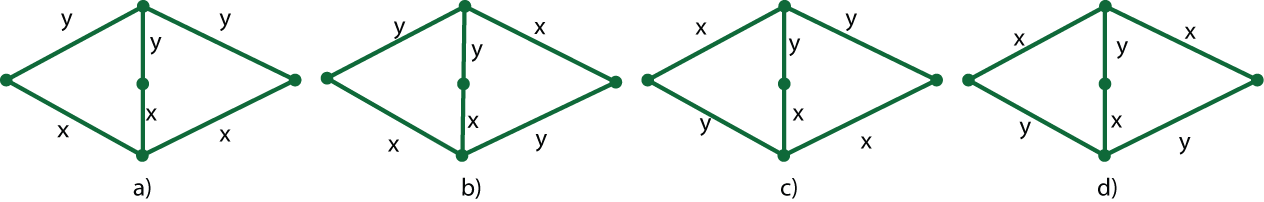}
\caption{}
\label{2step1}
\end{figure}

We construct a similar pattern of possibilities for the blue strings by labeling the arrows with $X$ and $Y$.
We need to determine which red and blue cases can be paired together. Since all of the labellings are distinct, the only potential pairs of cases are $(b,C)$ and $(c,B)$ (See Figure \protect\ref{2step2}).
\begin{figure}[H]
\includegraphics[scale=.8]{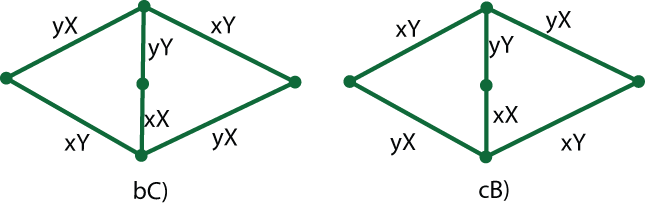}
\caption{}
\label{2step2}
\end{figure}

As the labelings are symmetric, we can assume without loss of generality that the diagram has edge labels of type $(b,C)$.  Since this subquiver is full there are no missing vertices.  This means that changes in chamber labels of the same cardinality indicate a unique braid crossing as pictured in Figure \ref{2cross}.  
\begin{figure}[H]
  \centering
  \subfloat[]{\label{2cross}\includegraphics[width=0.4\textwidth]{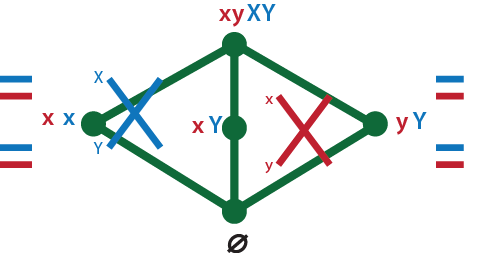}}  \hspace{.3in}
  \subfloat[]{\label{2rebuilt}\includegraphics[width=0.4\textwidth]{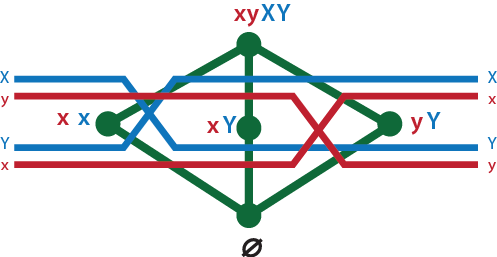}}
\caption{Reconstructed 2-move}
\label{2recon}
\end{figure}Since no other crossings may occur in this region we  connect the strings without creating any other crossings. Doing so yields the 2-move pictured in Figure \ref{2rebuilt}.  
\end{proof}

\section{Describing $\Phi_n$}
Using theorems \ref{thm:3} and \ref{thm:4} we graph $\Phi_n$ for $n\le 5$.  The smallest graph $\Phi_2$ consists of two vertices connected by an edge.
The graph of $\Phi_3$ first appeared in \protect\cite{TP}.  Figure \ref{phi3} shows a new representation of $\Phi_3$ which indicates the presence of a Hamiltonian path.
$\Phi_4 $ is pictured in Figure \ref{phi4} and Tables \ref{data}, \ref{degree} and \ref{degree5} summarize information about $\Phi_n$ .  
\begin{figure}[H]
\includegraphics[width=.6\textwidth]{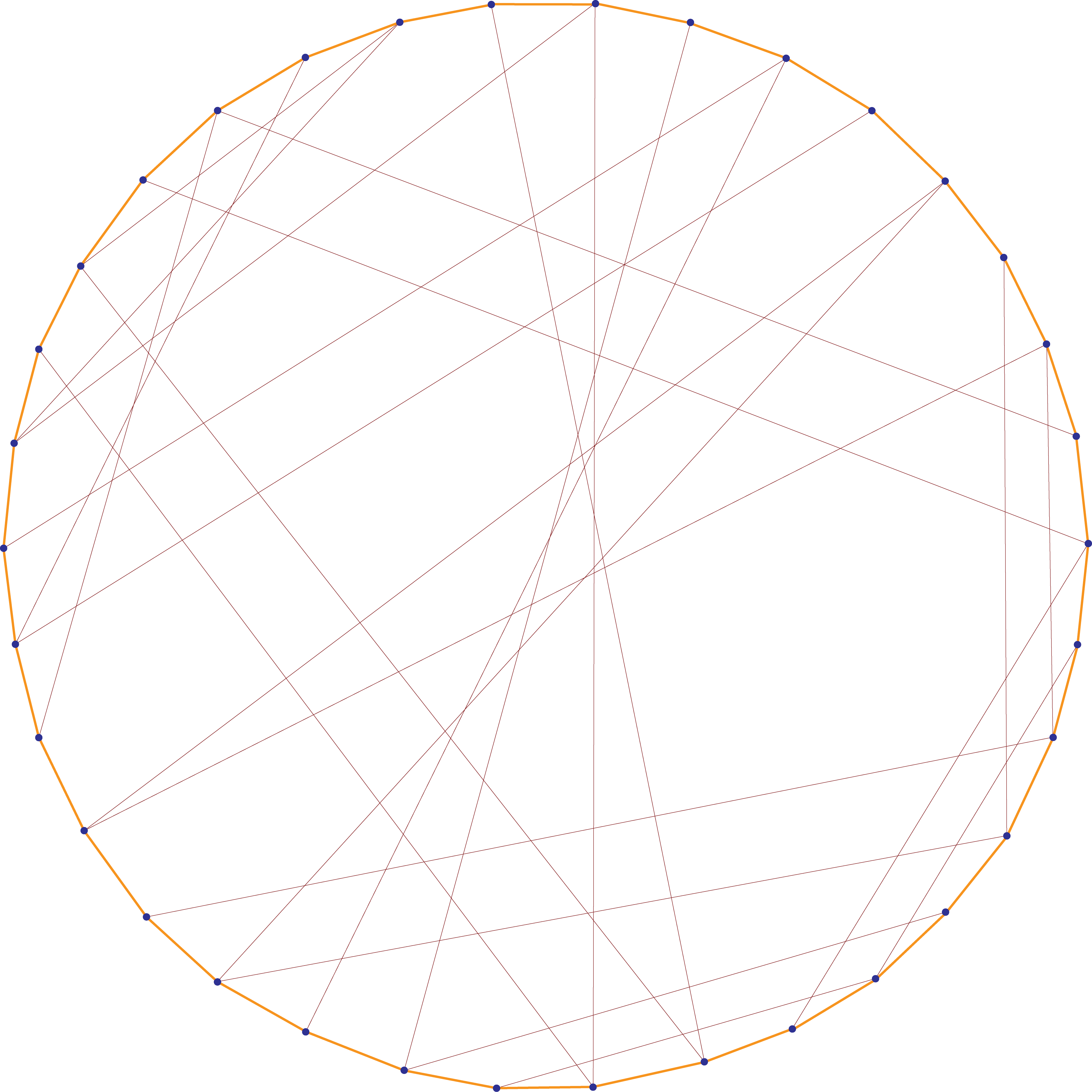}
\caption{$\Phi_3$ with Hamiltonian cycle highlighted}
\label{phi3}
\end{figure}
\begin{figure}[H]
\includegraphics[width=.85\textwidth]{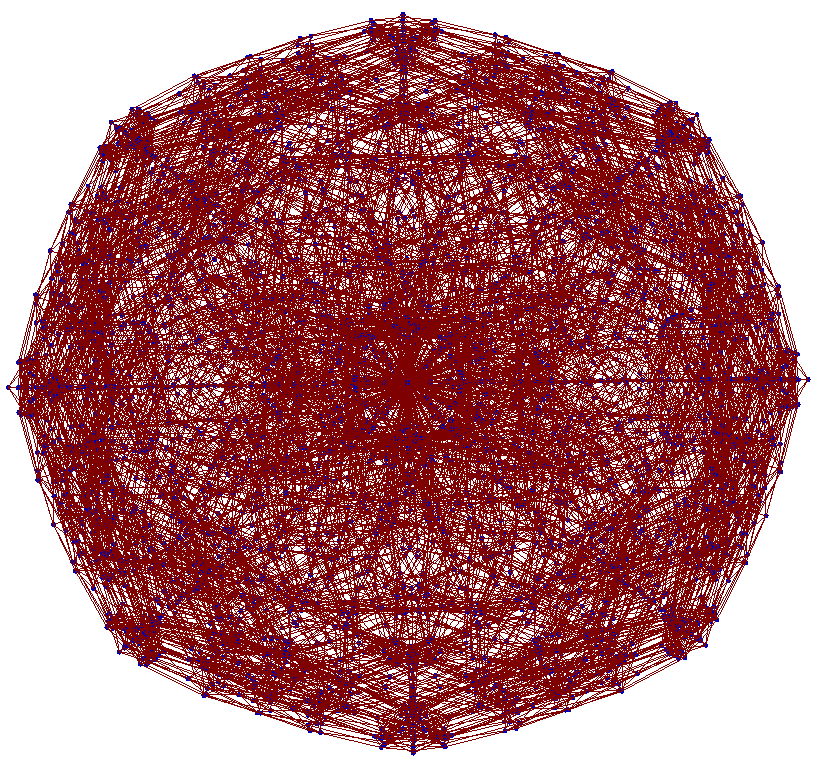}
\caption{Depiction of $\Phi_4$}
\label{phi4}
\end{figure}
\begin{table}[H]
\begin{tabular}{|l||c|c|c|c|}
\hline
 & $\Phi_2$ & $\Phi_3$ & $\Phi_4$ & $\Phi_5$ \\
\hline \hline
Vertices & 2 & 34 & 4894 & 5520372 \\
\hline
Edges & 1 & 120 & 33300& 60930112 \\
\hline
\end{tabular}
\caption{$\Phi_n$ edge and vertex data}
\label{data}
\end{table}
\begin{table}[H]
\begin{tabular}{|l||c|c|c|c|c|c|c|c|c|}
\hline
  & 1 & 2 & 3 & 4 & 5 & 6 & 7 & 8 & 9 \\
\hline \hline
$\Phi_2$ & 2 &  &  & &  &  & &  &   \\
\hline
$\Phi_3$ &  &  & 16  & 18 &  &  &  &  &  \\
\hline
$\Phi_4$ &  &  &   &  2 & 522 & 1362 & 1754 &  1054& 200\\
\hline
\end{tabular}
\caption{Number of vertices of given degree for $\Phi_n$}
\label{degree}
\end{table}
\begin{table}[H]
\begin{tabular}{|l||c|c|c|c|c|c|}
\hline
 & 6 & 7 & 8 & 9 & 10 & 11 \\
\hline
$\Phi_5$ & 84 & 28584 & 198596 & 632028 & 1165732 & 1402756   \\
\hline \hline
 & 12 & 13 & 14 & 15 & 16 & \\
\hline
$\Phi_5$  & 1165888  & 651188 & 227520 & 44452 & 3544 & \\
\hline
\end{tabular}
\caption{Number of vertices of given degree for $\Phi_5$}
\label{degree5}
\end{table}

\section{Total Positivity Conjecture}
\label{sec:n3}
\begin{definition}An $n\times n$ matrix $M$ with entries in $\bR$ is called \emph{totally positive} if all minors of $M$ are positive.
\end{definition}
\begin{definition} (Fomin and Zelevinsky \protect\cite{TP}) For each chamber label $(r,b)$ of $w$ we define the minor $\Delta_{r,b}$ to be the determinant of the matrix with rows of $M$  corresponding to $r$ and columns of $M$ corresponding to $b$. We call the collection of all such minors the \emph{chamber minors of $w$}.
\end{definition}
Fomin and Zelevinsky proved that for any commutation class of double wiring diagrams $w$, a matrix $M$ is totally positive if and only if all of its chamber minors are positive \protect\cite{TP}.
\begin{conjecture} (Fomin and Zelevinsky \protect\cite{TP}) Every minor can be written as a Laurent polynomial with positive coefficients in terms of the chamber minors of $w$.
\end{conjecture}
\begin{theorem} The Fomin Zelevinsky conjecture is true for $n\le 4$.
\end{theorem}
\begin{proof}
In \protect\cite{TP} Fomin and Zelevinsky show that if $w$ and $w'$ are linked by a braid move as pictured in Figure \ref{moves}, then their chamber minors satisfy the equation \[AD+BC=XY.\]
Using the program Fermat \protect\cite{fermat} and a C++ program written by the first author, we  verify this conjecture using the following algorithm:
\begin{enumerate}
\item For each vertex $v \in \Phi_n$ and minor $\Delta$ find a path from $v$ to a vertex $v'$ such that $\Delta$ is a chamber minor of $v'$.  This is possible since $\Phi_n$ is connected and
every minor appears as the chamber minor for some double wiring diagram.
\item At each edge of this path use Fermat to compute the new minor as a laurent polynomial in terms of the previous minors using the formula
$Y=(AD+BC)/Z$.  The Laurent Theorem \protect\cite{ltheorem} guarantees the result will be a laurent polynomial in the chamber minors of $v$.  Repeat the process until $\Delta$ is written as a laurent polynomial in terms of the chamber minors of $v$.
\item Verify that the corresponding laurent polynomial has all positive coefficients.  
\end{enumerate}
\end{proof}
\subsection{Example}
 We demonstrate that  $\Delta_{14,12}$ can be written as a Laurent polynomial in the chamber minors of the wiring diagram in Figure \ref{4wire} with positive coefficients.
\begin{enumerate}
\item The diagrams in Figure \protect\ref{path} determine a path in $\Phi_4$ from the wiring diagram to a vertex that contains $\Delta_{14,12}$ as a chamber minor. 
\begin{figure}[H]
\includegraphics[width=\textwidth]{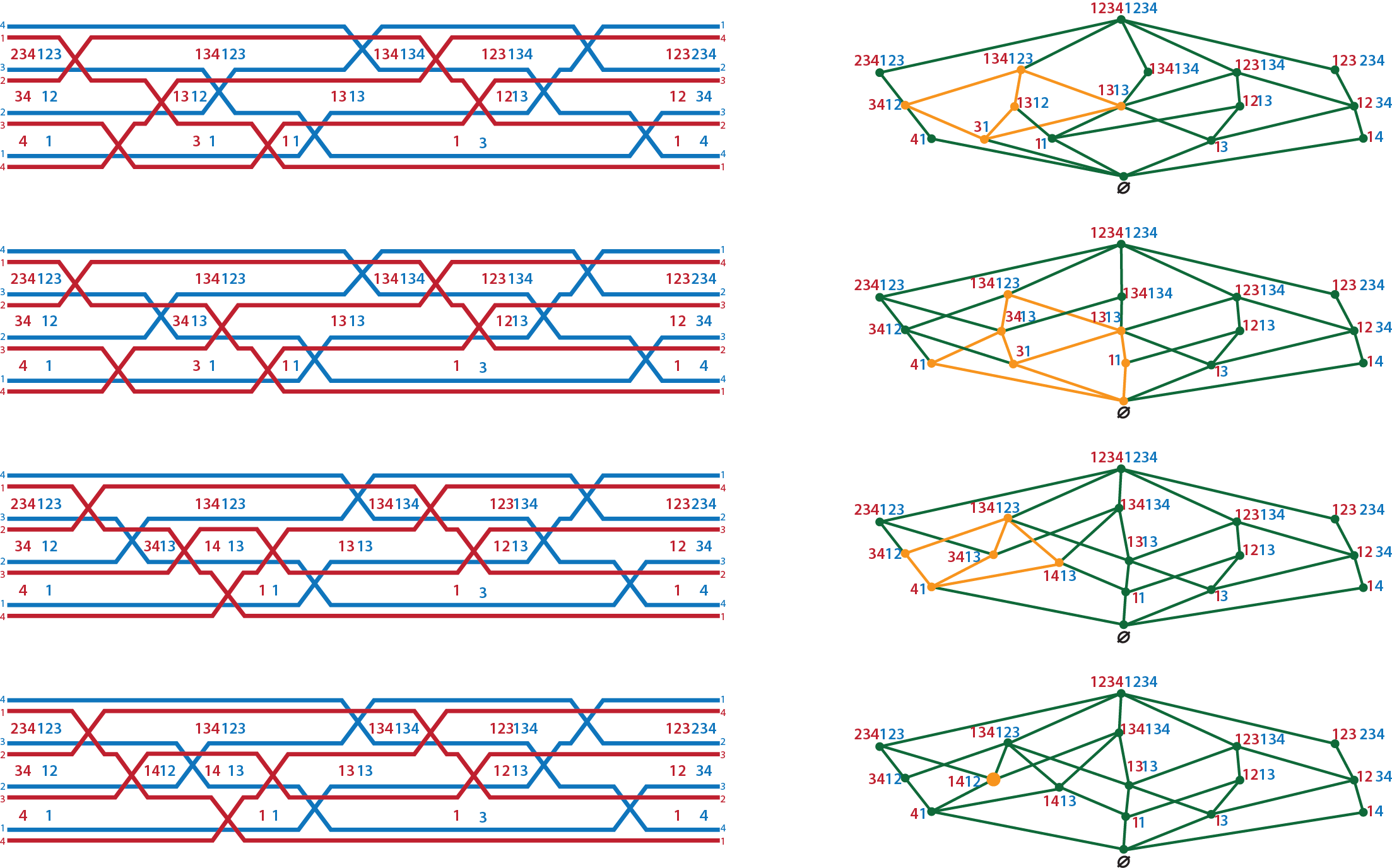}
\caption{Path to $\Delta_{14,12}$}
\label{path}
\end{figure}
\item At each step we use Fermat to compute the new laurent polynomial.
\begin{align*}
\Delta_{34,13} &= \frac{\Delta_{34,12}\Delta_{13,13} + \Delta_{134,123}\Delta_{3,1}}{\Delta_{13,12}} \\
 &= \Delta_{134,123}\Delta^{-1}_{13,12}\Delta_{3,1} + \Delta_{34,12}\Delta^{-1}_{13,12}\Delta_{13,13}\\
\\
\Delta_{14,13}& = \frac{\Delta_{1,1}\Delta_{34,13} + \Delta_{4,1}\Delta_{13,13}}{\Delta_{3,1}} \\
&=\Delta_{134,123}\Delta^{-1}_{13,12} + \Delta_{34,12}\Delta^{-1}_{13,12}\Delta_{13,13}\Delta^{-1}_{3,1}\Delta_{1,1} + \Delta_{13,13}\Delta_{4,1}\Delta^{-1}_{3,1} \\
\\
\Delta_{14,12}& = \frac{\Delta_{14,13}\Delta_{34,12} + \Delta_{134,123}\Delta_{4,1}}{\Delta_{34,13}}\\
&= \Delta_{34,12}\Delta^{-1}_{3,1}\Delta_{1,1} + \Delta_{13,12}\Delta_{4,1}\Delta^{-1}_{3,1} 
\end{align*}
\item Observe that the coefficients in the expression of $\Delta_{14,12}$, in terms of the chamber minors, are all positive.
\end{enumerate}

\begin{remark}
There are $34\times 14=476$ pairs of vertices and chamber minors in $\Phi_3$, and $62\times 4,894 =303,420$ such combinations in $\Phi_4$.
We are unable to confirm the conjecture for $n=5$ since it would involve $242 \times 5,520,372 =1,335,930,024$ computations involving extremely large laurent polynomials. When $n=4$ the laurent polynomials in question frequently  have over one hundred terms.
\end{remark}
\newpage
\bibliography{double}{}
\bibliographystyle{plain}

\end{document}